\documentclass[12pt,reqno,a4paper]{amsart}

\usepackage{amssymb,amsthm}
\usepackage{amsfonts,cmtiup,comment,stmaryrd}
\usepackage{hyperref}
\usepackage{mathrsfs,cmtiup}

\makeatletter
\def\blfootnote{\xdef\@thefnmark{}\@footnotetext}
\makeatother

\newtheorem{theorem}{Theorem}[section]
\newtheorem{lemma}[theorem]{Lemma}
\newtheorem{proposition}[theorem]{Proposition}

\theoremstyle{definition}
\newtheorem{definition}[theorem]{Definition}
\newtheorem{remark}[theorem]{Remark}

\numberwithin{equation}{section}

\begin{document}
\title{Fitting height of finite groups admitting a~fixed-point-free automorphism satisfying an~additional polynomial identity}

\author{E. I. Khukhro}
\address{Charlotte Scott Research Centre for Algebra, University of Lincoln, U.K.}
\email{khukhro@yahoo.co.uk}

\author{W. A. Moens$\dagger$}

\address{Faculty of Mathematics, University of Vienna, Austria}
\email{Wolfgang.Moens@univie.ac.at}

\keywords{Finite group; fixed-point-free automorphism; Fitting height; Hall--Higman type theorems}
\subjclass[2010]{}

\let\thefootnote\relax\footnote{$\dagger$The second author  died in May 2022.}

\begin{abstract}
Let  $f(x)$ be a non-zero polynomial with integer coefficients.
An automorphism $\varphi$ of a group $G$ is said to satisfy the elementary abelian identity $f(x)$ if the linear transformation induced by $\varphi$ on every characteristic elementary abelian section of $G$ is annihilated by $f(x)$. We prove that if a finite (soluble) group $G$ admits a fixed-point-free automorphism $\varphi$ satisfying an elementary abelian identity $f(x)$, where $f(x)$ is a primitive polynomial, then  the Fitting height of $G$ is bounded in terms of $\deg(f(x))$. We also prove that if $f(x)$ is any non-zero polynomial and $G$ is a $\sigma'$-group for a finite set of primes $\sigma=\sigma(f(x))$ depending only on $f(x)$, then the Fitting height of $G$ is bounded in terms of  the number $\operatorname{irr}(f(x))$ of different irreducible factors in the decomposition of $f(x)$. These bounds for the Fitting height are stronger than the well-known  bounds in terms of the composition length $\alpha (|\varphi|)$ of $\langle\varphi\rangle$ when $\deg f(x)$ or   $\operatorname{irr}(f(x))$ is small in comparison with $\alpha (|\varphi|)$.
\end{abstract}
\maketitle

\section{Introduction}

 An automorphism $\varphi$ of a group $G$ is said to be \emph{fixed-point-free} if $C_G(\varphi)=1$, that is, the only fixed point of $\varphi$ is $1$.  By the celebrated theorem of Thompson~\cite{tho59}, a finite group with a fixed-point-free automorphism of prime order $p$ is nilpotent (and the nilpotency class is bounded in terms of $p$ by Higman's theorem~\cite{hig}, with the bound made effective by Kreknin and Kostrikin~\cite{kre, kos-kre}). Based on the classification of finite simple groups, Rowley~\cite{row} proved the solubility of a finite group admitting a fixed-point-free automorphism of any order (not necessarily coprime to $|G|$).

If a finite (soluble) group $G$ admits a fixed-point-free automorphism $\varphi$ of coprime order, then the Fitting height of $G$ is at most the composition length $\alpha(|\varphi|)$ of $\langle \varphi\rangle$, which is the best possible bound. In this most general form, this is a special case of Berger's theorem~\cite{ber}. Earlier results under certain restrictions on the primes dividing  $|\varphi|$ were obtained by Shult~\cite{shu} and Gross~\cite{gro}. Their papers contained  important so-called non-modular Hall--Higman type theorems (with Gross also referring to Dade's seminar notes of 1964).

A special case of Dade's theorem~\cite{dad} gives  a bound for the Fitting height of $G$ admitting a fixed-point-free automorphism $\varphi$ of non-coprime order. The bound furnished by Dade's theorem is exponential in $\alpha(|\varphi|)$; a significant improvement to a quadratic bound was recently obtained by Jabara~\cite{jab}.

There are also many results, starting from the papers of Thompson~\cite{tho64} and Dade~\cite{dad}, on bounding the Fitting height of finite soluble groups in terms of fixed points and orders of their (not necessarily cyclic) groups of automorphisms. We refer the reader to the survey of Turull~\cite{tur}, who obtained some of the best results in this area.

Another important direction is studying groups with automorphisms satisfying certain identities. In particular, so-called splitting automorphisms $\varphi$ (satisfying the identity $xx^{\varphi} x^{\varphi^2}\cdots x^{\varphi^{|\varphi|-1}}=1 $) arise in connection with the Hughes subgroup and its  generalizations, and with periodic profinite groups~\cite{alp, bae, ers, esp, hug-tho, jab94, jab96, keg, khu80,khu86,khu89, khu91, khu94,  mak-khu,zel}.

In this paper we consider finite groups $G$ admitting a fixed-point-free automorphism $\varphi$ that satisfies an additional polynomial identity $f(x)$ for an arbitrary  polynomial $f(x)\in \mathbb{Z}[x]$;  see the precise definition below. Imposing this additional condition on the automorphism is justified by obtaining upper bounds for the Fitting height of $G$ independent  of $\alpha (|\varphi|)$. Instead, we obtain bounds for the Fitting height of $G$ depending only on  the degree of $f(x)$ (Theorem~\ref{t0}) or on the number $\operatorname{irr}(f(x))$ of different irreducible factors in the decomposition of $f(x)$
 (Theorem~\ref{t1}). The bounds obtained in our theorems are stronger than the known bound  $\alpha (|\varphi|)$ for coprime automorphisms (or a bound in terms of  $\alpha (|\varphi|)$ in general) when $\deg f(x)$  or   $\operatorname{irr}(f(x))$
  is small in comparison with $\alpha (|\varphi|)$. Our result in terms of $\deg (f(x))$ applies to any finite group $G$ (with a natural and unavoidable condition on the polynomial $f(x)$). The result  in terms of $\operatorname{irr}(f(x))$
  applies to all finite $\sigma'$-groups, where $\sigma=\sigma(f(x))$ is a finite set of primes depending only on $f(x)$.

Earlier this approach to the study of finite groups with fixed-point-free automorphisms was proposed by the second author \cite{moe0,moe1, moe2}.

We now pass to precise definitions and statements of the results. Let $\varphi$ be an automorphism of a group $G$.

\begin{definition}\label{d1}
	 We say that a polynomial $f(x) = a_0 + a_1 x + \cdots + a_d x^d \in \mathbb{Z}[x]$ is an \emph{ordered identity of $\varphi$} if $$g^{a_0} \cdot (g^{\varphi})^{a_1} \cdots (g^{\varphi^d})^{a_d}=1\qquad \text{for all }g\in G.$$
When $G$ is an abelian group, the order of the factors here is unimportant, and then we can simply say  that $f(x)$  is an identity of $\varphi$.
\end{definition}

\begin{definition}\label{d2}
We say that a polynomial $f(x)\in \mathbb{Z}[x]$ is an \emph{elementary abelian identity} of $\varphi$ if $f(x)$ is an identity of the automorphisms induced by $\varphi$ on every characteristic elementary abelian section of $G$. In this case, we also say that $\varphi$ \emph{satisfies the elementary abelian identity} $f(x)$.
\end{definition}

It is clear that an ordered identity of $\varphi$ is also an elementary abelian identity of $\varphi$, but the converse is not true in general.

\begin{remark}
 \label{r1}
  If $f(x)$ is an elementary abelian identity of $\varphi\in \operatorname{Aut}G$ and  $S$ is an elementary abelian $p$-group that is a characteristic  section of $G$, then the $\mathbb{F}_p$-linear transformation induced by $\varphi$ on $S$ regarded as a vector space over $\mathbb{F}_p$ is annihilated by the polynomial $f(x)$ (reduced modulo $p$) in the ordinary sense of linear algebra.
\end{remark}

We further recall that a polynomial $f(x) = a_0 + a_1 x + \cdots + a_d x^d \in \mathbb{Z}[x]$ is \emph{primitive} if  its \emph{content} $\gcd(a_0,a_1,\ldots,a_d)$ is $1$. We can now state our first result.

\begin{theorem}\label{t0}
	Suppose that a finite (soluble) group $G$ admits a fixed-point-free automorphism satisfying an elementary abelian identity  $f(x)\in \mathbb{Z}[x]$, where $f(x)$ is a primitive polynomial. Then the Fitting height of $G$ is at most $2 + 112 \cdot \deg(f(x))^2$.
\end{theorem}

Examples show that if we drop the condition that $f(x)$ is primitive, then one cannot obtain a bound for the Fitting height of $G$ in terms of $\deg(f(x))$. We conjecture that such a bound could possibly be  obtained for an arbitrary polynomial $f(x) \neq 0$ if ``elementary abelian identity'' is replaced with the stronger  ``ordered identity''.

A stronger bound for the Fitting height is obtained for an arbitrary polynomial $f(x)$  under an additional restriction on the prime divisors of the order of the group. Recall that if $\sigma$ is a set of primes, then a finite group $G$ is a \emph{$\sigma'$-group}  if $|G|$ is not divisible by any prime in $\sigma$.  Let  $\operatorname{irr}(f(x))$ denote the number of different irreducible factors of $f(x)$ in $\mathbb{Z}[x]$ (counted without multiplicities).

\begin{theorem}\label{t1}
	Suppose that a finite (soluble) group $G$ admits a fixed-point-free automorphism satisfying an elementary abelian identity  $f(x)\in \mathbb{Z}[x]$, where $f(x)$ is a non-zero polynomial. There is a finite set of primes $\sigma=\sigma(f(x))$ depending only on $f(x)$ such that if $G$ is a $\sigma'$-group, then the Fitting height of $G$ is at most $2 + \operatorname{irr}(f(x))^2$.
\end{theorem}

The set of primes $\sigma(f(x))$ in this theorem is described explicitly in Definition~\ref{d-invariants}.

Earlier the second author~\cite{moe1} proved that, for ordered identities $f(x)$ with irreducible polynomial $f(x)$, Theorem~\ref{t1} holds with a sharp bound $1$ for the Fitting height  (for a different finite set of primes $\sigma(f(x))$). In~\cite{moe2} the second author proved a result similar to Theorem~\ref{t1} bounding the Fitting height of $G$ by the number of irreducible factors of $f(x)$ counting multiplicities   for a certain class of polynomials $f(x)$. The bound for the Fitting height that we obtain in Theorem~\ref{t1} may not be best-possible, but importantly it only depends on the number of different irreducible factors $\operatorname{irr}(f(x))$ and the theorem holds for any non-zero polynomial $f(x)$. The bound for the Fitting height in Theorem~\ref{t0} is weaker, in terms of $\deg (f(x))$, but the advantage is that it does not impose any restrictions on the prime divisors of $|G|$ provided that the polynomial $f(x)$ is primitive (for example, when $f(x)$ is monic).

Let $G$ be a finite (soluble) group admitting a fixed-point-free automorphism $\varphi$ satisfying an  elementary abelian  identity $f(x)\in \mathbb{Z}[x]$. The first step in the proof of both Theorems~\ref{t0} and~\ref{t1} is to use Hall--Higman type theorems (with certain modifications) for obtaining a reduction to the situation where $\varphi$ has  order bounded in terms of $\operatorname{deg}(f(x))$. Then the proof of Theorem~\ref{t0} follows by an application of a special case of  Dade's theorem~\cite{dad}, or rather Jabara's~\cite{jab} recent improvement for the bound for the Fitting height of a finite group admitting a fixed-point-free automorphism of not necessarily coprime order.

In the proof of Theorem~\ref{t1} it is the Shult--Gross--Berger theorem that is ultimately applied after a reduction to the case of a (possibly different) coprime automorphism  such that the number $\alpha (|\varphi|)$ of prime factors in $|\varphi|$ counting multiplicities is bounded in terms of $\operatorname{irr}(f(x))$.
This becomes possible after defining a `forbidden' finite set of primes $\sigma$ depending only on $f(x)$ such that for $\sigma '$-groups $G$  the automorphism $\varphi$ can be assumed to be of coprime order and `non-exceptional' in the sense of Hall--Higman type theorems, while $f(x)$ can be assumed to be a product of  cyclotomic polynomials. Then Hall--Higman type theorems are applied again to reduce to a situation where $\alpha (|\varphi|)$ is bounded in terms of $\operatorname{irr}(f(x))$ and an application of the Shult--Gross--Berger theorem finishes the proof.

\section{Preliminaries}\label{s-prel}

Suppose that a group $A$ acts by automorphisms on a group $B$. We use the usual notation for commutators $[b,a]=b^{-1}b^a$ and commutator subgroups $[B,A]=\langle [b,a]\mid b\in B,\;a\in A\rangle$, as well as for centralizers $C_B(A)=\{b\in B\mid b^a=b \text{ for all }a\in A\}$
and $C_A(B)=\{a\in A\mid b^a=b\text{ for all }b\in B\}$. In particular, then $A/C_A(B)$ embeds in the automorphism group $\operatorname{Aut}B$. If $\varphi $ is an automorphism of a group $G$ and $S$ is a $\varphi$-invariant section of $G$, then we denote by  $\varphi|_S$ the automorphism induced by $\varphi$ on $S$;  sometimes we denote the induced automorphism by the same letter when this causes no confusion.

We recall the well-known property of coprime actions.

\begin{lemma}\label{l-copr}
	Let $\varphi$ be an automorphism of coprime order of a finite group~$G$. If $N$ is a normal $\varphi$-invariant subgroup, then $C_{G/N}(\varphi)=C_G(\varphi)N/N$. In particular, if $\varphi$ acts trivially on every factor of some subnormal $\varphi$-invariant series of $G$, then $\varphi=1$.
\end{lemma}

It is also known that if $\varphi$ is a coprime automorphism of a finite group $G$, then for every prime $q$ there is a $\varphi$-invariant Sylow $q$-subgroup. It is important for us that a similar property and an analogue of Lemma~\ref{l-copr} hold for fixed-point-free automorphisms of any order, not necessarily coprime to the order of the group. Henceforth we freely use the fact that a finite group with a fixed-point-free automorphism is soluble~\cite{row}.

\begin{lemma}[{\cite[Theorem~10.1.2]{gor}}] \label{l-hall}
	Let $G$ be a finite (soluble) group admitting a fixed-point-free automorphism $\varphi$, and let $\sigma$ be a set of primes. Then $G$ has a $\varphi$-invariant Hall-$\sigma$ subgroup.
\end{lemma}

\begin{proof}
	We only need to replace Sylow's theorem with Hall's theorem in the proof of~\cite[Theorem~10.1.2]{gor} to obtain the lemma for arbitrary sets of primes $\sigma$.
\end{proof}

\begin{lemma}[{\cite[Lemma~10.1.3]{gor}}] \label{l-fpf}
	Let $G$ be a finite group and let $\varphi$ be a fixed-point-free  automorphism of $G$. If $N$ is a normal $\varphi$-invariant subgroup, then  $\varphi$ induces a fixed-point-free   automorphism of the quotient $G/N$.
\end{lemma}

The following result in its most  general form is a special case of Berger's theorem~\cite{ber}. Earlier results under certain restrictions on the primes dividing  $|\varphi|$ were obtained by Shult~\cite{shu} and Gross~\cite{gro}.  Recall that $\alpha(n)$ denotes the number of prime divisors of $n$ counting multiplicities.

\begin{theorem}[Shult--Gross--Berger]
  \label{t-sgb}
If a finite (soluble) group $G$ admits a fixed-point-free automorphism of coprime order $n$, then the Fitting height of $G$ is at most $\alpha (n)$.
\end{theorem}

For fixed-point-free automorphisms of non-coprime order, the first (exponential) bound for the Fitting height was obtained as a special case of a theorem by Dade~\cite{dad}; the quadratic bound was recently obtained by Jabara~\cite[Corollary 1.2]{jab}.

\begin{theorem}[Dade--Jabara]
	\label{t-j}
	If a finite (soluble) group $G$ admits a fixed-point-free automorphism of order $n$, then the Fitting height of $G$ is bounded in terms of $\alpha (n)$ and is at most $7 \alpha (n)^2$.
\end{theorem}

We collect in the next lemma well-known facts about minimal polynomials of linear transformations (which are assumed monic). If $\varphi$ is an automorphism of a vector space $V$ over a field $F$, we regard $V$ as a right $F\langle\varphi\rangle$-module.

\begin{lemma}\label{l1}
Suppose that $\varphi$ is an automorphism of a vector space $V$ over a field~$K$.

\begin{enumerate}

  \item[\rm (a)] If $\varphi$ is regarded as a linear transformation of the vector space $V\otimes _KK_1$ obtained by extending the ground field to $K_1$, then the minimal polynomial remains the same.

  \item[\rm (b)] If $\varphi$ is diagonalizable, then the degree of the minimal polynomial of $\varphi$ is equal to the number of different eigenvalues of $\varphi$.

  \item[\rm (c)] There is a vector $v\in V$ such that the minimal polynomial of $\varphi$ for $v$, that is, the polynomial $g(x)$ of smallest degree such that $vg(\varphi)=0$, is the same as  the minimal polynomial of $\varphi$ for the whole space $V$.

  \item[\rm (d)] The degree of the minimal polynomial of $\varphi$ is equal to the maximum dimension of the subspace spanned by an orbit of a vector under the action of $\langle\varphi\rangle$.

  \item[\rm (e)]  If $V=V_1\oplus\dots\oplus V_k$, where $V_i\varphi=V_{i+1}$ for $i=1,\dots,k-1$ and $V_k\varphi=V_1$, then the degree of the minimal polynomial of $\varphi$ on $V$ is $k\cdot d$, where $d$ is the degree of the minimal polynomial of the restriction $\varphi^k{|_{V_1}}$.

        \item[\rm (f)] The degree of the minimal polynomial of any power $\varphi ^k$ is at most the degree of the minimal polynomial of $\varphi$.
\end{enumerate}
\end{lemma}

\begin{proof} All these properties are well-known, but we still indicate some references.

  (a)  See~\cite[Ch.~6, page~192]{hof-kun}.

  (b) This is well known.

  (c) See~\cite[Ch.~7, Corollary of Theorem~3, page~237]{hof-kun}.

  (d) This follows from (c).

  (e) This follows from (d).

  (f) This follows from (d).
\end{proof}

Recalling Definition~\ref{d1} we can say that a linear transformation $\varphi$ of a vector space satisfies an identity $f(x)\in \mathbb{Z}[x]$ if $\varphi$ is annihilated by $f(x)$ (reduced in the ground field).

\begin{lemma}\label{l-power}
	Let $\varphi$ be an automorphism of finite order $|\varphi |$ of a vector space $V$ over a field of characteristic $q$ coprime to $|\varphi|$. Suppose that $\varphi $ satisfies the identity $f(x)=\prod _{i=1}^{c}\Phi_{n_i}(x)$, where the $\Phi_{n_i}(x)$ are some cyclotomic polynomials such that $\gcd(q,n_i) = 1$.
	\begin{itemize}
		\item[\rm (a)] For any $s\in \mathbb{N}$ the power $\varphi^s$ satisfies the identity $\prod _{i=1}^{c}\Phi_{{n_i}/{\operatorname{gcd}(n_i,s)}}(x)$.
		\item[\rm (b)] If $|\varphi|$ is divisible by a power of a prime $p^k$, then one of the integers $ n_i$ is divisible by $p^k$.
	\end{itemize}
\end{lemma}

\begin{proof}
	Let $\widetilde V$ be the vector space obtained from $V$ by extension of the ground field to an algebraically closed field $K$, so that $\varphi$ naturally becomes a $K$-linear transformation of $\widetilde V$. Since the order of $\varphi$ is coprime to the characteristic, the transformation $\varphi$ of finite order is diagonalizable over $K$.  The eigenvalues of $\varphi$ are roots of $f(x)$ (regarded as a polynomial in $K[x])$.
	
	(a) The eigenvalues of $\varphi^s$ are the $s$-th powers of the roots of $\varphi$, and therefore they are roots of $\prod _{i=1}^{c}\Phi_{{n_i}/{\operatorname{gcd}(n_i,s)}}(x)$. Hence $\varphi^s$ satisfies the polynomial $\prod _{i=1}^{c}\Phi_{{n_i}/{\operatorname{gcd}(n_i,s)}}(x)$ on $\widetilde V$ and therefore also on $V$.
	
	(b) If $|\varphi|$ is divisible by a power of a prime $p^k$, then at least one of the eigenvalues of $\varphi$ has multiplicative order divisible by $p^k$ and therefore can be a root of a cyclotomic polynomial $\Phi_{n_i}(x)$ only if $n_i$ is divisible by $p^k$.
\end{proof}

\begin{lemma} \label{l-ideal}
	Let $\varphi$ be an automorphism of a group $G$. Then the elementary abelian identities of $\varphi$ form an ideal of $\mathbb{Z}[x]$.
\end{lemma}

\begin{proof}
	Let $f(x), g(x) \in \mathbb{Z}[x]$ be elementary abelian identities of $\varphi$, and let $h(x)\in \mathbb{Z}[x]$.
Let $S$ be a characteristic elementary abelian section of $G$ regarded as a right $\mathbb{Z}\langle\varphi\rangle$-module. Then by definition $x(f(\varphi )+g(\varphi))=0+0=0$ and  $x(f(\varphi )h(\varphi)) = 0h(\varphi)=0$  for all $x \in S$. Thus, both  $f(x)+ g(x)$ and $f(x)h(x)$ are  elementary abelian identities of $\varphi$.
\end{proof}

\begin{remark}
	In what follows, we will consider polynomials in the ring $\mathbb{Z}[x]$ and also polynomials with coefficients in a field of prime characteristic. Unless specifically stated otherwise, every polynomial is assumed to belong to $\mathbb{Z}[x]$. Moreover: divisors, greatest common divisors, decompositions into irreducible factors, etc. are to be understood in the ring $\mathbb{Z}[x]$, unless stated otherwise.
\end{remark}

The prime number theorem, commonly attributed to Hadamard and de la Vallée Poussin, establishes that the prime counting function $\pi(x)$ satisfies $\pi(x) \sim x / \ln(x)$,  that is,  $\lim_{x \rightarrow \infty} \pi(x) / (x / \ln(x)) = 1$. We will use the following estimate due to Rosser and Schoenfeld~\cite[Corollary 1]{ros-sch} that is valid for \emph{all} $x > 1$.

\begin{theorem} \label{tpnt}
	For all $x> 1$, we have $\pi(x) < 1.25506 \cdot x / \ln(x)$.
\end{theorem}

\section{Hall--Higman type theorems}\label{s-hh}

In this section we lay out the foundations of further proofs by producing Hall--Higman type theorems, both of `modular' and `non-modular' kind.  First we state the celebrated Hall--Higman Theorem~B, which is of `modular' kind, dealing with the minimal polynomial of an element of order $p^m$ in a linear group over a field of characteristic~$p$.

\begin{theorem}[Hall and Higman {\cite[Theorem~B]{ha-hi}}] \label{t-hh}
Let $H$ be a $p$-soluble linear group over a field of characteristic $p$, with no normal $p$-subgroup greater than~$1$. If $g$ is an element of order $p^m$ in $H$,
then the minimal polynomial of $g$ is $(x-1)^d = 0$, where $d = p^m$, unless there is
an integer $m_0 \leqslant m$ such that $p^{m_0}-1$ is a power of a prime $q$ for
which a Sylow $q$-subgroup of $H$ is non-abelian, in which case, if $m_0$ is the least such integer, $p^{m-m_0}(p^{m_0}-1)\leqslant d\leqslant p^m$.
\end{theorem}

Note that the degree $d$ of the minimal polynomial of $g$ satisfies $d\geqslant p^{m}-p^{m-1}$ in all cases.

We now consider Hall--Higman type results in so-called `non-modular' situations, which analyse the minimal polynomial of an element of order $p^m$ in a linear group over a field of characteristic not equal to~$p$.

\begin{lemma}
  \label{l2m}
  Suppose that $L$ is an extra-special $r$-group of
  order $r^{2t+ 1}$, and $\langle\eta\rangle$ is a cyclic group of order $p^k$ for a prime $p\ne r$  acting (not necessarily faithfully) by automorphisms on $L$ such that the induced automorphism $\bar \eta$ of $L$ has order $p^m$, acts regularly on the set $L/Z(L)\setminus\{1\}$ with all orbits of length $p^m$, and centralizes $Z(L)$. Suppose that the semidirect product $L\langle\eta\rangle$ acts by linear transformations on a vector space $V$ over an algebraically closed field $K$ whose characteristic does not divide $p\cdot r$, and suppose  that $V$ is an irreducible $KL\langle\eta\rangle$-module and a faithful and homogeneous $KL$-module.
  \begin{itemize}
    \item[ {\rm (a)}] Then
  $\eta$  as a linear transformation of $V$ has at least $p^m-1$ different eigenvalues, so that the minimal polynomial of $\eta$ on $V$ has degree at least $p^m-1$.

	\item[ {\rm (b)}] If in addition $k=m$ (that is, $\eta$ acts faithfully on $L$), then either the minimal polynomial of $\eta$ on $V$ is $x^{p^k}-1$ or $r^t = p^k-1$ and if $p = 2$, then $t = 1$.
  \end{itemize}
  \end{lemma}

  \begin{proof} Part (b) is the well-known result going back to the work of Dade, Gross, Shult, who modified the Hall--Higman Theorem~B in~\cite{ha-hi} for the  `non-modular' case; see, for example,~\cite[Theorem~2.2]{gro}  or~\cite[Satz~V.17.13]{hup}.

  Part (a) does not seem to have appeared in the literature; its proof is similar to the well-known proof for (b), which was really using only the number of eigenvalues of $\eta$. We write this modified proof in full for the benefit of the reader.

  By~\cite[Lemma~IX.2.5]{hup-bla2},  $V$ is an irreducible $KL$-module and has dimension $\dim_K V = r^t$ as a vector space over $K$. Then the enveloping algebra $E$ for $L$ coincides with the full matrix algebra~\cite[Theorem~3.6.2]{gor}  and $\dim_K  E = r^{2t}$. The elements of
$Z (L)$ are represented by scalar transformations and multiplication by such a transformation in $E$ is equivalent
to multiplication by the corresponding field element. Therefore any set of representatives
of the $r^{2t}$ cosets of $Z(L)$ forms a basis of the algebra $E$.

The element $\eta$ naturally acts on $E$ and we calculate the dimension of the centralizer of $\eta$ in $E$. Since $\eta^{p^m}$ belongs to the centre of the semidirect product $L\langle\eta\rangle$ by hypothesis, the subgroup $\langle\eta^{p^m}\rangle$ is represented by scalar  transformations and the action of $\langle\eta\rangle$ on $E$ factors through to the action of $\langle\bar\eta\rangle=\langle\eta\rangle/\langle\eta^{p^m}\rangle$. Since $\bar\eta$ acts regularly on $L/Z(L)\setminus\{1\}$ with orbits of length $p^m$, it acts regularly on the elements of some basis of $E$  except for one element of this basis that corresponds
to $Z(L)$ and belongs to $C_E(\eta)$. In every subspace of $E$ spanned by a non-trivial orbit of $\eta$ on this basis, the fixed point subspace is one-dimensional (spanned by  the sum of the elements of the orbit).  It follows that
$$\dim C_E (\eta ) = \frac{r^{2t}-1}{p^m}+1.$$
We now calculate the same quantity in another way. Let $a_i$, $i = 1,\dots , l$, be the multiplicities of the $l$
distinct eigenvalues of the linear transformation $\eta$. Then the matrix of $\eta$  in some basis is block-diagonal
consisting of $l$ scalar blocks with different eigenvalues on the diagonals. The centralizer of this matrix in the full matrix algebra consists of all block-diagonal matrices with the same block-partition. Therefore,
$$\dim C_E (\eta ) = \sum_{i=1}^{l}a_i^2.$$
Thus,
$$ \frac{r^{2t}-1}{p^m}+1=\sum_{i=1}^{l}a_i^2.$$
If $l= p^m$, there is nothing to prove. Suppose that $l\leqslant p^m-1$. Then, since  $\sum_{i=1}^{l}a_i=r^t$, we have
$$
\sum_{i=1}^{l}a_i^2\geqslant l\cdot \left(\frac{r^t}{l}\right)^2=\frac{r^{2t}}{l}\geqslant \frac{r^{2t}}{p^m-1},
$$
and as a result,
$$
\frac{r^{2t}-1}{p^m}+1=\sum_{i=1}^{l}a_i^2\geqslant \frac{r^{2t}}{p^m-1}.
$$
This is equivalent to saying that $r^t + 1 \leqslant p^m$. Using this inequality we now obtain that
$$r^t=\sum_{i=1}^{l}a_i \leqslant \sum_{i=1}^{l}a_i^2 = \frac{r^{2t}-1}{p^m}+1=\frac{(r^{t}-1)(r^t+1)}{p^m}+1\leqslant r^t.$$
Therefore, all the inequalities are actually equalities, in particular, $l= p^m-1$. Thus, $l\geqslant p^m-1$ in all cases, as required.
  \end{proof}

  The following lemma readily follows from Lemma~\ref{l2m} and is a further  variation on `non-modular'  Hall--Higman type theorems. Similarly to Lemma~\ref{l2m}, part (b) about the case of faithful action is a well-known result; see~\cite[Theorem~3.1]{shu} and~\cite[Theorem~4.1]{gro}.

  \begin{lemma}\label{l3}
   Suppose that $r$ is a prime, $R$ is an $r$-group, and $\langle\psi\rangle$ is a cyclic group of order $p^k$ for a prime $p\ne r$  acting (not necessarily faithfully) by automorphisms on $R$ such that the induced automorphism of $R$ has order $p^m$. Suppose that the semidirect product $R\langle\psi\rangle$ acts by linear transformations on a vector space $V$ over a field $K$ whose characteristic does not divide $p\cdot r$, and suppose that $V$ is a faithful $KR$-module.
  \begin{itemize}
    \item[ {\rm (a)}] Then
  the minimal polynomial of $\psi$ on $V$ has degree at least $p^m-p^{m-1}$.

	\item[ {\rm (b)}] If in addition $k=m$ that is, $\psi$ acts faithfully on $R$, then either the minimal polynomial of $\psi$ on $V$ is $x^{p^k}-1$, or there exist positive integers $k_0 \leqslant k$ and $t$ such that $r^t = p^{k_0}-1$ and if $p = 2$, then $t = 1$.
  \end{itemize}
  \end{lemma}

  Note that in the case of  $r^t = p^{k_0}-1$ either $r=2$ or $p = 2$ and $r$ is a Mersenne prime.

  \begin{proof} (a) We can assume that the field $K$ is algebraically closed, since the hypotheses of the lemma remain valid for $R\langle\psi\rangle$ regarded as a group of linear transformations of the vector space obtained by extending the ground field, while the minimal polynomial of $\psi$ will be the same by Lemma~\ref{l1}(a).

  Choose a minimal $\psi$-invariant $r$-subgroup $M$ of $R$ on which $\psi$ acts with order $p^m$, that is, on which $\psi ^{p^{m-1}}$ acts non-trivially. Then $M$ is either an elementary abelian or a non-abelian special $r$-group, $\psi$ acts irreducibly on $M/[M,M]$, and $\psi ^{p^{m-1}}$ acts non-trivially on $M /[M,M]$ and trivially on $[M,M]$ (see~\cite[Theorem~5.3.7]{gor}). Since $\psi$ acts irreducibly on $M/[M,M]$,  in particular, $C_M(\psi ^{p^{m-1}})=[M,M]$ so that $M=[M,\psi ^{p^{m-1}}]$, and all orbits of $\psi$ on $M/[M,M]\setminus\{1\}$ have length $p^m$.

 Since the characteristic of the field $K$ is coprime to $p\cdot r$, the $KM\langle\psi\rangle$-module $V$ is completely reducible. Let $W$ be an irreducible $KM\langle\psi\rangle$-submodule of $V$ on which the subgroup $M=[M,\psi ^{p^{m-1}}]$ acts non-trivially. Applying Clifford's theorem~\cite[Theorem~3.4.1]{gor} with respect to the normal subgroup $M$ we decompose $W=W_1\oplus\dots\oplus W_{p^s}$ into a sum of homogeneous $KM$-submodules, which are transitively permuted by $\langle\psi\rangle$ with $\langle\psi^{p^s}\rangle$ being the stabilizer of $W_1$ in $\langle\psi\rangle$. Then  $\langle\psi^{p^s}\rangle$  is also the stabilizer of every $W_i$ since $\langle\psi\rangle$ is abelian. Note that $p^s\leqslant p^m$. Let $L$ be the image of $M$ in its action on~$W_1$. Then $W_1$ is an irreducible
$K L\langle\psi^{p^s}\rangle$-module and a homogeneous $KL$-module.

If $L$ is abelian, then it is cyclic and central, so that $[L,\psi^{p^s}]=1$ and therefore $[L^{\psi^i},\psi^{p^s}]=1$ for all $i$ and $[L,\psi^{p^s}]$ acts trivially on $W$, whence $p^s=p^m$. Then  $\psi $ has minimal polynomial on $V$ of degree at least $p^m$ by Lemma~\ref{l1}(e).

If $L$ is not abelian, then it is an extra-special $r$-group, $\psi ^{p^s}$ acts trivially on $Z(L)$ and has regular orbits of length $|\psi ^{p^s}|/|\psi ^{p^m}|=p^{m-s}$ on $L/Z(L)$. Thus  the semidirect product $L \langle\psi^{p^s}\rangle$ and the $K L\langle\psi^{p^s}\rangle$-module $W_1$ satisfy the hypotheses of
Lemma~\ref{l2m}(a), by which the minimal polynomial of  $\psi^{p^s}$ on $W_1$ has degree at least $p^{m-s}-1$. If $m > s$, then by Lemma~\ref{l1}(e) the minimal polynomial of $\psi$ has degree at least $p^s(p^{m-s}-p^{m-s-1})=p^{m}-p^{m-1}$. If $m = s$, then the same lemma gives the lower bound $p^s \cdot 1 = p^m \geqslant p^m - p^{m-1}$. Thus, the degree of the minimal polynomial of  $\psi$ on $V$ is at least $p^{m}-p^{m-1}$ in all cases, and part (a) is proved.

 (b) This part is basically known from the papers of Shult and Gross, see~\cite[Theorem~3.1]{shu} and~\cite[Theorem~4.1]{gro}, but for completeness we also derive this result here from Lemma~\ref{l2m}(b).
 Assuming $k=m$ (that is, $\psi$ acts faithfully on $R$), we repeat the above arguments. When $L$ is abelian, we obtain that $\psi $ has minimal polynomial on $V$ of degree at least $p^k$ by Lemma~\ref{l1}(e), so it must be $x^{p^k}-1$. When $L$ is extraspecial,  the semidirect product $L \langle\psi^{p^s}\rangle$ and the $K L\langle\psi^{p^s}\rangle$-module $W_1$ satisfy the hypotheses of
Lemma~\ref{l2m}(b), by which either the minimal polynomial of  $\psi^{p^s}$ has degree at least $p^{m-s}$ or $r^t = p^{k-s}-1$ and if $p = 2$, then $t = 1$. In the first case, by Lemma~\ref{l1}(e) the minimal polynomial of  $\psi$ has degree at least $p^sp^{k-s}=p^{k}$ and then it must be  $x^{p^k}-1$. The second case corresponds to the other alternative in part (b).
  \end{proof}

\section{Reduction to an automorphism of bounded order.\\ Proof of Theorem~\ref{t0}.}\label{s-weak}

In this section we perform a reduction of the proofs of Theorems~\ref{t0} and~\ref{t1} to the case where the order of a fixed-point-free automorphism $\varphi$ satisfying an elementary abelian identity $f(x)\in \mathbb{Z}[x]$ is bounded in terms of $\deg (f(x))$. At the end of this section we finish the proof of Theorems~\ref{t0}.

Since in Theorems~\ref{t0} and~\ref{t1} we need to bound the Fitting height and $F(G)=\bigcap _qO_{q',q}(G)$,
it is sufficient to bound the Fitting height of $G/O_{q',q}(G)$ for every prime~$q$. Here $O_{q'}$ is the largest normal $q'$-subgroup, and $O_{q',q}(G)$ is the inverse image of the largest normal $q$-subgroup of $G/O_{q'}(G)$.

\begin{proposition}\label{pr-weak}
Let $G$ be a (soluble) finite  group admitting a fixed-point-free automorphism $\varphi$ satisfying an elementary abelian identity $f(x)\in \mathbb{Z}[x]$ such that $f(x)$ does not vanish modulo any prime divisor of $|G|$. Let $q$ be a prime dividing $|G|$, and let $\bar G=G/O_{q',q}(G)$. Then the order of the automorphism $\varphi|_{\bar G/F(\bar G)}$  induced by $\varphi$ on $\bar G/F(\bar G)$ is bounded in terms of $\deg(f(x))$, namely,
$$
| \varphi|_{\bar G/F(\bar G)} | \leqslant (2 \deg(f(x)))^{2 \deg(f(x))}.
$$
Moreover, the number $\alpha(|\varphi|_{\bar G/F(\bar G)}|)$ of prime divisors of $| \varphi|_{\bar G/F(\bar G)} |$ counting multiplicities  is at most $ 4 \cdot \deg(f(x))$.
\end{proposition}

 Of course, a crude bound for $\alpha(|\varphi|_{\bar G/F(\bar G)}|)$ immediately follows from a bound for $|\varphi|_{\bar G/F(\bar G)}|$, but we included a sharper bound, which is easily obtained in the proof.

Note that the condition that the polynomial $f(x)$ does not vanish modulo prime divisors of $|G|$ is automatically satisfied in Theorem~\ref{t0}. Moreover, the definition of the set $\sigma(f(x))$ will ensure that this condition is also satisfied in Theorem~\ref{t1}. In what follows, when we will be using the fact that $\varphi$ satisfies $f(x)$ on some elementary abelian characteristic section, it will only matter that $f(x)$ reduced modulo some prime is non-zero, so the minimal polynomial of $\varphi$ regarded as a linear transformation of this section has degree at most $\deg(f(x))$.

\begin{proof}
 The quotient $\bar G=G/O_{q',q}(G)$ acts faithfully on the Frattini quotient $V$ of $O_{q',q}(G)/O_{q'}(G)$. Let  $\bar\varphi$ be the induced automorphism of $V\bar G$, which is fixed-point-free by Lemma~\ref{l-fpf}. Let $p$ be a prime divisor of $|\bar\varphi |$ and let $|\bar\varphi |=s_pp^{k_p}$ for $s_p$ coprime to $p$. For brevity let us write simply $s=s_p$ and $k=k_p$ when we focus on this prime $p$. Note that it may happen that $p=q$ or $p\ne q$.

We write $\psi =\bar\varphi^s$, which is  a generator of the Sylow $p$-subgroup of $\langle\bar\varphi\rangle$.
Let $p^m$  be the order of the automorphism induced by $\psi $ on a $\bar\varphi$-invariant Hall $p'$-subgroup $H$  of~$\bar G$, which exists by Lemma~\ref{l-hall}.

\begin{lemma}\label{l-order}
The order of the automorphism of $\bar G/F(\bar G)$  induced by $\psi$ is at most~$p^m$; in particular, if  $\psi $ centralizes a $\bar\varphi$-invariant Hall $p'$-subgroup of~$\bar G$, then $\psi$ acts trivially on $\bar G/F(\bar G)$.
\end{lemma}
\begin{proof}
  We claim that $\psi^{p^m}\in O_p(\bar G\langle\bar\varphi\rangle)\leqslant F(\bar G\langle\bar\varphi\rangle)$. Indeed,
    by definition $\psi^{p^m}$  centralizes a $\varphi$-invariant  Hall $p'$-subgroup $H$ of $\bar G$, as well as, obviously, the Hall $p'$-subgroup of~$\langle \bar{\varphi}  \rangle$. Hence $\psi^{p^m}$  centralizes a Hall $p'$-subgroup $H_1$ of $\bar G \langle \bar{\varphi}  \rangle$. Since  $\bar G \langle \bar{\varphi}  \rangle=H_1P$, where $P$ is a Sylow $p$-subgroup containing  $\psi^{p^m}$,  all conjugates of $\psi^{p^m}$ belong to $P$ and therefore generate a normal $p$-subgroup contained in $O_p(\bar G\langle\bar\varphi\rangle)$. Thus, $\psi^{p^m}\in F(\bar G\langle\bar\varphi\rangle)$, and since $F(\bar G)=\bar G\cap F(\bar G\langle\bar\varphi\rangle)$,  the order of the automorphism of $\bar G/F(\bar G)$  induced by $\psi$ is at most~$p^m$.
\end{proof}

\begin{lemma}\label{l-pm}
The number $p^m$ is bounded in terms of $\operatorname{deg}(f(x))$, namely, $p^m \leqslant 2 \deg(f(x))$.
\end{lemma}
\begin{proof}
By Lemma~\ref{l-hall} for every prime divisor $r$ of $|H|$ there is a $\bar\varphi$-invariant Sylow $r$-subgroup of $H$. Clearly, $\psi$ must act as an automorphism of order $p^m$ on at least one such Sylow subgroup. Let $R$ be a $\psi$-invariant Sylow $r$-subgroup of $\bar G$ for some prime $r\ne p$ on which $\psi$ acts with order $p^m$, that is, on which $\psi ^{p^{m-1}}$ acts non-trivially. The semidirect product $R\langle\psi\rangle$ acts by linear transformations on $V$, and $V$ is a faithful $\mathbb{F}_qR$-module. We claim that $p^m - p^{m-1} \leqslant \deg(f(x))$, which implies the required bound.

First suppose that  $r\ne q$. Then the action of the semidirect product $R\langle\psi\rangle$  on $V$ gives rise to a Hall--Higman type situation, which may be `modular' when  $p=q$, or `non-modular' when $p\ne q$.

In the `non-modular' case $p\ne q$ we apply Lemma~\ref{l3}(a), by which the minimal polynomial of $\psi$ on $V$ has degree at least $p^m-p^{m-1}$. By Lemma~\ref{l1}(f) the  minimal polynomial of $\varphi$ on $V$ must also have degree at least $p^m-p^{m-1}$. Since $\varphi$ satisfies $f(x)$ on $V$, we obtain that
 $p^m-p^{m-1}\leqslant \operatorname{deg}(f(x))$.

In the `modular' case  $p= q$, we apply Theorem~\ref{t-hh} after certain preparation.
 Since $\bar G=G/O_{p',p}(G)$ and $V$ is the Frattini quotient of $O_{p',p}(G)/O_{p'}(G)$, while $\langle\psi \rangle$ is a $p$-group, we have $O_{p'}(V\bar G\langle\psi \rangle)=1$. Since $\langle\psi \rangle /\langle\psi ^{p^m}\rangle$ acts faithfully on a Hall $p'$-subgroup of $\bar G$, while $\psi ^{p^m}\in O_{p}(\bar G\langle\psi \rangle)$ as shown in the proof of Lemma~\ref{l-order}, we further obtain
 $$O_{p',p}(V\bar G\langle\psi \rangle)=O_{p}(V\bar G\langle\psi \rangle)=V\langle\psi^{p^m}\rangle.$$

Let $U$ be the Frattini quotient  of $O_{p',p}(V\bar G\langle\psi \rangle)/O_{p'}(V\bar G\langle\psi \rangle)$, on which the group $V\bar G\langle\psi \rangle/O_{p',p}(V\bar G\langle\psi \rangle)$ acts faithfully by conjugation satisfying the hypotheses of Theorem~\ref{t-hh}.
Since the order of the automorphism of $U$ induced by $\psi$ is $p^m$, by Theorem~\ref{t-hh} the minimal polynomial of $\psi$ on $U$ has degree at least $p^m-p^{m-1}$. We have
$$
U=(V/[V,\psi ^{p^m}])\times (\langle\psi ^{p^m}\rangle/\langle\psi ^{p^{m+1}}\rangle).
$$
Since both factors of $U$ are $\psi$-invariant and $\langle\psi ^{p^m}\rangle/\langle\psi ^{p^{m+1}}\rangle$ is centralized by $\psi$, the minimal polynomial of $\psi$ on the first factor, and therefore on $V$, must also have degree at least $p^m-p^{m-1}$. By Lemma~\ref{l1}(f) the  minimal polynomial of $\bar\varphi$ on $V$ must have degree at least $p^m-p^{m-1}$. Since $\bar\varphi$ satisfies $f(x)$ on the characteristic abelian section $V$, we obtain that
 $p^m-p^{m-1}\leqslant \operatorname{deg}(f(x))$.

Now suppose that $r=q$. In this case $R$ is a $q$-group, and we consider the action of the semidirect product $R\langle\psi\rangle$  on $F(\bar G)$, which is a nilpotent $q'$-group containing its centralizer. Since $R$ acts faithfully on $F(\bar G)$, there is some Sylow $t$-subgroup $T$ of $F(\bar G)$ for $t\ne r$ on which the subgroup $[R, \psi ^{p^{m-1}}]$ acts non-trivially. Let $\widetilde R=R/C_R(T)$. Note that by the choice of $T$ the automorphism induced by $\psi$ on $\widetilde R$ has order $p^m$. Then the action of the semidirect product $\widetilde R\langle\psi\rangle$  on the Frattini quotient $U$ of $T$ again gives rise to a Hall--Higman type situation, which may be `modular' when  $p=t$, or `non-modular' when $p\ne t$.
Using the same Hall--Higman type arguments as above, based either on Lemma~\ref{l3}(a) or on Theorem~\ref{t-hh}, we obtain in similar fashion that $p^m - p^{m-1} \leqslant \deg(f(x))$.

So in all of the above cases, we have shown that $ p^m - p^{m-1} \leqslant \deg(f(x))$, and this implies the required bound. Indeed, if $m = 0$, then $p^m \leqslant 2 \deg(f(x))$. If $m = 1$, then $p^m \leqslant p^{m-1} + \deg(f(x)) = 1 + \deg(f(x)) \leqslant 2 \deg(f(x))$. If $m \geqslant 2$, then $p^m = p p^{m-1} \leqslant 2 (p-1) p^{m-1} \leqslant 2 \deg(f(x))$.
\end{proof}

We return to the proof of the proposition. Let $p_1^{m_1} \cdots p_l^{m_l}$ be the prime factorization of $|\varphi{|_{\bar G/F(\bar G)}}|$ with distinct prime divisors $p_1, \ldots , p_l$ and corresponding multiplicities $m_1 , \ldots , m_l > 0$. Lemmas~\ref{l-order} and~\ref{l-pm} then give us the bound
$$
p_i^{m_i} \leqslant 2 \deg(f(x))
$$
for each $i \in \{1,\ldots,l\}$, and therefore  the bound $|\varphi{|_{\bar G/F(\bar G)}}| \leqslant (2 \deg(f(x)))^{2 \deg(f(x))}$.

Furthermore,  Theorem~\ref{tpnt} gives us the bound
$$
l \leqslant \pi(2 \deg(f(x))) < 1.25506  \cdot (2 \deg(f(x))) / \ln(2 \deg(f(x))),
$$
where $\pi(\cdot)$ is the usual prime counting function. For each $i \in \{1,\ldots,l\}$, we also have the obvious bound $m_i \leqslant \log_{2}(2 \deg(f(x))) = \ln(2)^{-1} \cdot \ln(2 \deg(f(x)))$. Altogether, we obtain
\begin{align*}
	\alpha(|\varphi{|_{\bar G/F(\bar G)}}|) &= m_1 + \cdots + m_l \\
&\leqslant \ln(2)^{-1} \cdot \ln(2 \deg(f(x)))\cdot l\\
	&\leqslant  \ln(2)^{-1} \cdot \ln(2 \deg(f(x))) \cdot 1.25506  \cdot (2 \deg(f(x))) / \ln(2 \deg(f(x)))  \\
	&\leqslant  4 \cdot \deg(f(x)).\tag*{\qedhere}
\end{align*}
\end{proof}

\begin{proof}[Proof of Theorem~\ref{t0}]
	Recall that $G$ is a finite (soluble) group admitting a fixed-point-free automorphism $\varphi$ satisfying an elementary abelian identity $f(x)\in \mathbb{Z}[x]$, where $f(x)$ is a primitive polynomial, and we need to obtain a bound for the Fitting height of $G$ in terms of $\operatorname{deg}(f(x))$.  Let $q$ be any prime dividing $|G|$ and define $D := (G/O_{q',q}(G))/F(G/O_{q',q}(G))$. By Proposition~\ref{pr-weak}, we have $\alpha(|\varphi{|_D}|) \leqslant 4 \deg(f(x))$. Then by Dade's theorem~\cite{dad} the Fitting height $h(D)$ of $D$ is bounded in terms of $\alpha (|\varphi{|_{D}}|)$, and Jabara's paper~\cite{jab} gives the bound $7 \cdot \alpha(|\varphi{|_{D}}|)^2 $. So, for each such prime $q$, we obtain $h(D) \leqslant 112 \cdot \deg(f(x))^2$ and therefore $h(G/O_{q',q}(G)) \leqslant 1 + 112 \cdot \deg(f(x))^2$. Since $F(G)=\bigcap_q O_{q',q}(G)$, we obtain the required bound $h(G) \leqslant 2 + 112 \cdot \deg(f(x))^2$ for the Fitting height of $G$.
\end{proof}

\section{Proof of Theorem~\ref{t1}}\label{s-gsb}

In Theorem~\ref{t1}, for a finite soluble group $G$ with a fixed-point-free automorphism $\varphi$ satisfying an elementary abelian identity $f(x)\in \mathbb{Z}[x]$, we need to obtain a bound on the Fitting height of $G$ in terms of $\operatorname{irr}(f(x))$. As a first step, we use Proposition~\ref{pr-weak} to perform a reduction to the case where the order of $\varphi$  is bounded in terms of $\deg (f(x))$. Then, roughly speaking, we can pass to a fixed-point-free automorphism $\varphi$ satisfying the elementary abelian identity $\gcd(f(x),\,  x^{|\varphi |}-1)$, which is a product of cyclotomic polynomials the number of which is at most $\operatorname{irr}(f(x))$. First we deal with this situation in the following proposition, which requires imposing additional conditions on the prime divisors of $|G|$. Then we finish the proof of Theorem~\ref{t1} after defining a finite set of  `forbidden' prime divisors of $|G|$ depending only on $f(x)$.

\begin{proposition}\label{pr-gsb}
	Let $G$ be a finite (soluble) group admitting a fixed-point-free automorphism $\varphi$ satisfying an elementary abelian identity $f(x)$ that is a product  of $c$ cyclotomic polynomials: $f(x) = \Phi_{n_1}(x) \cdots \Phi_{n_c}(x)$. Suppose that
$$  \gcd(|G|,2)=\gcd(|G|,\,|\varphi|!) = \gcd(|G|, \,n_1 \cdots n_c) = 1.
$$
Then the Fitting height $h(G)$ of $G$ is at most $ c^2$.
\end{proposition}

Note that the coprimeness conditions on the order of $G$  ensure that $\varphi$ is a coprime automorphism of $G$ for which no exceptional situations arise in the `non-modular' Hall--Higman type arguments for automorphisms induced by powers of~$\varphi$.

\begin{proof}
We proceed by induction on the pairs $(c,\alpha(|\varphi|))$ ordered lexicographically. We formally include the case $c = 0$ as the base of the induction, which means that $f(x)=1$. By the definition of elementary abelian identity we then have $u=1$ for every $u\in S$ in every characteristic abelian section $S$ of $G$, which of course means that $G = 1$, so that indeed $h(G) = 0 \leqslant 0^2$. We may therefore assume that $c \geqslant 1$.

First we show that the number of distinct primes dividing the order of $\varphi$ can be assumed to be at most $c$, by way of possibly replacing $\varphi$ with another fixed-point-free automorphism. Suppose that for some prime $p$ the cyclotomic  polynomial $\Phi_{p}(x)$ does not occur in the factorization of $f(x)$. Then, since the eigenvalues of $\varphi$ on every elementary abelian characteristic section $S$ of $G$ are roots of $f(x)$ and $\operatorname{gcd}(|G|,n_i)=1$ for every factor $\Phi_{n_i}(x)$ of $f(x)$, the power $\varphi^p$ is also a fixed-point-free automorphism of~$S$. Therefore $\varphi^p$ is a fixed-point-free automorphism of $G$. By Lemma~\ref{l-power}(a), the automorphism $\varphi^p$ satisfies the elementary abelian identity $\prod _{i=1}^{c}\Phi_{{n_i}/{\operatorname{gcd}(n_i,p)}}(x)$ with the same number of cyclotomic factors (or possibly fewer if repeats appeared). Clearly, the coprimeness hypotheses of Proposition~\ref{pr-gsb} on $|G|$ also hold with respect to $\varphi^p$ instead of $\varphi$. Since $\alpha(|\varphi^p|)= \alpha (|\varphi |) -1$, induction completes the proof.  Hence we can assume from the outset that for every prime $p$ dividing $|\varphi|$, the cyclotomic polynomial $\Phi_{p}(x)$ does occur in the factorization of $f(x)$, and that therefore the number of distinct primes dividing $|\varphi|$ is at most $c$.

For every prime $q$ dividing $|G|$,  consider $\bar G =G/O_{q',q}(G)$ as above. It is sufficient to prove that $h(\bar G) \leqslant c^2-1$ for every $q$, as then the Fitting height of $G$ will be at most $c^2$ because $F(G)=\bigcap_qO_{q',q}(G)$. We fix the prime $q$ for what follows.
The group $\bar G$ acts faithfully on the Frattini quotient $V$ of $O_{q',q}(G)/O_{q'}(G)$. Let  $\bar\varphi$ be the induced automorphism of $V\bar G$. Let $p$ be a prime divisor of $|\bar\varphi |$ and let $|\bar\varphi |=s_pp^{k_p}$ for $s_p$ coprime to~$p$. For brevity we write  $s=s_p$ and $k=k_p$ when we focus on this prime $p$, and let $\psi =\bar\varphi^s$ be  a generator of the Sylow $p$-subgroup of $\langle\bar\varphi\rangle$.

There must be some factor  $\Phi_{n_j}(x)$ of  $f(x)$ with $p^k$ dividing $n_j$, since $p^k$ obviously divides the order of $\varphi$. Indeed, the (cyclic) Sylow $p$-subgroup of $\langle\varphi\rangle$ must act faithfully on some characteristic elementary abelian section $S$ of $G$ by Lemma~\ref{l-copr}.
Then the automorphism induced by $\varphi$ on $S$ has some eigenvalues of order divisible by $p^k$ and therefore at least one of the factors $\Phi_{n_j}(x)$ of  $f(x)$ must have $n_j$ divisible by $p^k$ by Lemma~\ref{l-power}(b).

If the automorphism induced by $\psi $ on $\bar G/F(\bar G)$ has order less than $p^k=|\psi |$, then $\varphi$ has no primitive roots of unity of order divisible by $p^k$ on any elementary abelian characteristic section of $\bar G/F(\bar G)$. Then the polynomial $f(x) / \Phi_{n_j}(x)$, with $n_j$ divisible by $p^k$, is an elementary abelian identity of the automorphism of $\bar G/F(\bar G)$ that is induced by $\varphi$. Since $f(x) / \Phi_{n_j}(x)$ has $c-1$ cyclotomic factors, by the induction hypothesis the Fitting height of $\bar G/F(\bar G)$ is at most $(c-1)^2$, so that the Fitting height of $\bar G$ is at most $(c-1)^2+1$ and therefore at most $c^2$, as required, since $c\geqslant 1$.

Therefore we can assume that the automorphism induced by $\psi $ on $\bar G/F(\bar G)$ has the same order $p^k$ as on $V \bar G$. This condition means that an application of the Hall--Higman type  arguments will result in the minimal polynomial of $\psi $ being $x^{p^k}-1$, as will follow from Lemma~\ref{l3}(b) in view of the absence of exceptional situations.

Namely, choose a $\psi$-invariant Sylow $r$-subgroup $R$ of $\bar G$ for some prime $r\ne p$ on which $\psi$ acts with order $p^k$, that is, on which $\psi ^{p^{k-1}}$ acts non-trivially. The semidirect product $R\langle\psi\rangle$ acts by linear transformations on $V$, and $V$ is a faithful $\mathbb{F}_qR$-module.

First suppose that  $r\ne q$. Then the action of the semidirect product $R\langle\psi\rangle$  on $V$ gives rise to a `non-modular' Hall--Higman type situation, since $p\ne q$ by the hypotheses on $|G|$. Since $\psi$ acts faithfully on $R$, we can apply Lemma~\ref{l3}(b), by which the minimal polynomial of $\psi$ on $V$ is $x^{p^k}-1$ in view of the  hypotheses on $|G|$.

Now suppose that $r=q$. In this case, $R$ is a $q$-group, and we consider the action of the semidirect product $R\langle\psi\rangle$  on $F(\bar G)$, which is a nilpotent $q'$-group containing its centralizer. Since $R$ acts faithfully on $F(\bar G)$, there is some Sylow $t$-subgroup $T$ of $F(\bar G)$ for $t\ne r$ on which the subgroup $[R, \psi ^{p^{k-1}}]$ acts non-trivially. Let $\widetilde R=R/C_R(T)$. Note that by the choice of $T$ the automorphism induced by $\psi$ on $\widetilde R$ has order $p^k$. Then the action of the semidirect product $\widetilde R\langle\psi\rangle$  on the Frattini quotient $U$ of $T$ again gives rise to a `non-modular' Hall--Higman type situation, since $p\ne t$. Applying Lemma~\ref{l3}(b)  we obtain in similar fashion that the minimal polynomial of $\psi$ on $T/\Phi (T)$ is $x^{p^k}-1$.

Thus, in any case, there is an elementary abelian characteristic section $S$ of $V\bar G$ on which  the minimal polynomial of $\psi$  is $x^{p^k}-1$. We regard $S$ as a vector space over a finite field, which we extend to an algebraically closed one. Since the automorphism $\bar\varphi$ is of order coprime to the characteristic, it is diagonalizable.  Since the minimal polynomial of $\psi$  on $S$ is $x^{p^k}-1$, the eigenvalues of  $\psi=\bar \varphi^s$ are all primitive $p^i$-th roots of unity for all $i=0,1,2,\dots ,k$. At the same time, these eigenvalues are the $s$-th powers of the eigenvalues of $\bar\varphi$. Hence the eigenvalues of $\bar\varphi$ include primitive roots of unity for which the highest power of $p$ dividing their order ranges over all values $p^0,p^1,p^2,\dots ,p^k$. Since $\bar\varphi$ satisfies  the polynomial $f(x)=\prod_{i=1}^{c}\Phi _{n_i}(x)$, there must be different cyclotomic factors $\Phi_{n_i}(x)$ in $f(x)$ such that the highest power of $p$ dividing the $n_i$ ranges over all values $p^0,p^1,p^2,\dots ,p^k$.  Therefore the exponent $k=k_p$ of the highest power of $p$ dividing $|\bar\varphi|$ satisfies $k\leqslant c - 1$.

Since the number of distinct primes dividing $|\varphi|$  is at most $c$ by our assumption, we obtain that $\alpha(|\bar\varphi|)\leqslant c(c-1)$. By  the Shult--Gross--Berger Theorem~\ref{t-sgb}, the Fitting height of $V\bar G=V(G/O_{q',q})$ is at most $c(c-1)$, so obviously  the Fitting height of $G/O_{q',q}$ is also at most $c(c-1)$. Since this is true for every $q$, the Fitting height of $G/F(G)$ is also at most  $c(c-1)$. As a result, the Fitting height of $G$ is at most  $c(c-1)+1\leqslant c^2$, as required, since $c\geqslant 1$.
\end{proof}

We now introduce the finite set of primes $\sigma(f(x))$ that is used in Theorem~\ref{t1}. Let $\operatorname{Res}(a(x),b(x))$ be the {resultant} of polynomials $a(x),b(x) \in \mathbb{Z}[x]$. Recall that $\operatorname{Res}(a(x),b(x))$ is an integer, which is equal to $0$ exactly when $\deg(\gcd(a(x),b(x)))>0$. Note also that $\operatorname{Res}(a(x),b(x))=a(x)a_1(x)+b(x)b_1(x)$ for some $a_1(x),b_1(x)\in \mathbb{Z}[x]$ and therefore $\operatorname{Res}(a(x),b(x))$ belongs to the ideal generated by $a(x)$ and $b(x)$.

\begin{definition} \label{d-invariants}
	Let $f(x) = a_0 + a_1 x + \cdots + a_d x^d \in \mathbb{Z}[x]$ be a non-zero polynomial of degree $d := \deg(f(x))$  with content $a := \gcd(a_0 , \ldots , a_d)$. Using the auxiliary polynomials $u(x) := x^{(2d)^{2d}!}-1$  and $ v(x) := \gcd(f(x),u(x))$ we define the (non-zero) integer
$$\rho :=\operatorname{Res}(f(x)/v(x),\,u(x)/v(x)). $$
	We then define $\sigma(f(x))$ to consist of the prime divisors of the (non-zero) integer
$$
(2d)^{2d}! \cdot a \cdot \rho .
$$
Recall that we denote by $\operatorname{irr}(f(x))$ the number of different irreducible divisors of $f(x)$, counted without multiplicity.
\end{definition}

We will only need the following four properties of these invariants.

\begin{lemma} \label{l-resultant}
	Let $f(x) \in \mathbb{Z}[x] \setminus \{0\}$ and let $p$ be a prime not in $\sigma(f(x))$. Then the following hold.
	\begin{enumerate}
		\item[\rm (a)] The prime $p$ is odd and $p> (2d)^{2d}$.
		\item[\rm (b)] The polynomial $f(x)$ does not vanish modulo $p$.
		\item[\rm (c)] If $v(x)\neq 1$, then $v(x) = \Phi_{n_1}(x) \cdots \Phi_{n_c}(x)$ with $\gcd(p,n_1 \cdots n_c) = 1$ and $c \leqslant \operatorname{irr}(f(x))$.
		\item[\rm (d)] The polynomial $\rho \cdot v(x)$ is in the ideal of $\mathbb{Z}[x]$   generated by $f(x)$ and $u(x)$.
	\end{enumerate}
\end{lemma}

\begin{proof}
 (a) This is true because $p$ is coprime to $(2d)^{2d}!$.

(b) This holds because $p$ is coprime to the content $a$ of $f(x)$.

(c) This follows from the definitions of $u(x) =  x^{(2d)^{2d}!}-1$ and $v(x)= \gcd(f(x),u(x))$ and part~(a).

(d) Note that $\rho$ is in the ideal of $\mathbb{Z}[x]$ generated by $f(x)/v(x)$ and $u(x)/v(x)$, and therefore  the polynomial $\rho \cdot v(x)$ is in the ideal of $\mathbb{Z}[x]$  generated by $f(x)$ and $u(x)$.
\end{proof}

\begin{proof}[Proof of Theorem~\ref{t1}]
	Recall that $G$ is a finite (soluble) group admitting a fixed-point-free automorphism $\varphi$ satisfying an elementary abelian identity $f(x)\in \mathbb{Z}[x] \setminus \{0\}$ and that $G$ is a $\sigma(f(x))'$-group. We need to show that $h(G) \leqslant 2 + \operatorname{irr}(f(x))^2$. We continue to use the notation of Definition~\ref{d-invariants} and we distinguish between two cases: $v(x) \neq 1$ and $v(x) = 1$.
	
	We first consider the case $v(x) \neq 1$. According to Lemma~\ref{l-resultant}(c), there are nonnegative integers $c$ and $n_1 , \ldots , n_c$ such that $v(x) = \Phi_{n_1}(x) \cdots \Phi_{n_{c}}(x)$ and $c \leqslant \operatorname{irr}(f(x))$.  Let $q$ be any prime divisor of $|G|$ and define the quotients ${\bar G} := G/O_{q',q}(G)$ and $D := \bar{G} / F(\bar{G})$, as before. It now suffices to verify  that $v(x)$ is an elementary abelian identity of the fixed-point-free automorphism $\varphi{|_D}$ of $D$  induced by $\varphi$ and that
$$
 \gcd(|D|,2) = \gcd(|D|,|\varphi{|_D}|!) = \gcd(|D|, n_1 \cdots n_c) = 1.
$$
Indeed, then we will be able to apply Proposition~\ref{pr-gsb} to $D$, $\varphi{|_D}$, and $v(x)$ in order to obtain the bound $h(D) \leqslant c^2$ and therefore $h(\bar{G}) \leqslant 1 + c^2$. Since $F(G) = \bigcap_q O_{q',q}(G)$, we will then obtain the required bound $h(G) \leqslant 2 + c^2 \leqslant 2 + \operatorname{irr}(f(x))^2$ on the Fitting height of $G$, since $c\leqslant \operatorname{irr}(f(x))$.
	
	We first verify the coprimeness conditions.  Clearly, $ \gcd(|D|,2)=1$, since $ G$ is a $\sigma(f(x))'$-group. According to Lemma~\ref{l-resultant}(b), the polynomial $f(x)$ does not vanish modulo any prime divisor of $|G|$. Therefore we can apply Proposition~\ref{pr-weak} to $G$, $\varphi$, and $f(x)$ in order to obtain the bound $|\varphi{|_D}| \leqslant (2d)^{2d}$. By Lemma~\ref{l-resultant}(a), every prime divisor $p$ of $|D|$ satisfies $p > (2d)^{2d}$. Hence, $\gcd(p,|\varphi{|_D}|!) = 1$. By Lemma~\ref{l-resultant}(c), we also have $\gcd(p, n_1 \cdots n_c) = 1$.
	
	To prove that $v(x)$ is an elementary abelian identity of $\varphi{|_D}$, we first recall that $f(x)$ is elementary abelian identity of $\varphi{|_D}$. Since $|\varphi{|_D}| \leqslant (2d)^{2d}$, the polynomial  $u(x)= x^{(2d)^{2d}!}-1$ is also an elementary abelian identity of $\varphi{|_D}$. By Lemma~\ref{l-ideal}, the elementary abelian identities of $\varphi{|_D}$ form an ideal of $\mathbb{Z}[x]$.  By Lemma~\ref{l-resultant}(d), the polynomial $\rho \cdot v(x)$ belongs to the ideal generated by $ f(x)$ and $u(x)$;  therefore $\rho \cdot v(x)$  is an elementary abelian  identity of $\varphi{|_D}$.
We have $\gcd(|D|,\rho) =1$, since $\rho$ is a product of primes in $\sigma(f(x))$ and $G$ is a $(\sigma(f(x)))'$-group by hypothesis. In accordance with Euclid's algorithm, we have $1 \in \rho \cdot \mathbb{Z}[x] + |D| \cdot \mathbb{Z}[x]$ and therefore also $v(x) \in \rho \cdot v(x) \cdot \mathbb{Z}[x] + |D| \cdot v(x) \cdot \mathbb{Z}[x]$. Clearly, $|D|$ is an elementary abelian identity of~$\varphi{|_D}$. As a result, $v(x)$ is in the ideal of elementary abelian identities of~$\varphi{|_D}$, and the case $v(x)\ne 1$ is complete as explained above.
	
	Finally, we consider the case $v(x) = 1$.  Suppose that, for some prime divisor $q$ of $|G|$, the quotient $D := (G/O_{q',q}(G))/(F(G/O_{q',q}(G)))$ is non-trivial. Then we again have the bound $|\varphi{|_D}| \leqslant (2d)^{2d}$ by Proposition~\ref{pr-weak}. The polynomial $u(x)$ is therefore again an elementary abelian identity of $\varphi{|_D}$, so that $v(x)$ is again an elementary abelian identity of $\varphi{|_D}$. Since $v(x) = 1$, the group $D$ is therefore trivial. This contradiction shows that, for every prime divisor $q$ of $|G|$, we have $h(D) = 0$ and therefore $h(G/O_{q',q}(G)) \leqslant 1$. Since $F(G) = \bigcap_{q}O_{q',q}(G)$, we obtain the required bound $h(G) \leqslant 2 \leqslant 2 + \operatorname{irr}(f(x))^2$ on the Fitting height of $G$.
\end{proof}

\section*{Acknowledgements}

The research of the second author was supported by the Austrian Science Fund (FWF) P30842 -- N35.

\end{document}